\documentclass[12pt]{article}

\usepackage{amsmath}
\usepackage{amssymb}
\usepackage{amsthm}
\usepackage{latexsym}
\usepackage{cite}
\usepackage{psfrag}
\usepackage{epsfig}
\usepackage{graphicx}
\usepackage[latin1]{inputenc}

\newtheorem{theorem}{Theorem}[section]
\newtheorem{lemma}[theorem]{Lemma}
\newtheorem{corollary}[theorem]{Corollary}
\newtheorem{proposition}[theorem]{Proposition}
\theoremstyle{definition}
\newtheorem{definition}[theorem]{Definition}
\newtheorem{remark}[theorem]{Remark}

\makeatletter \@addtoreset{equation}{section} \makeatother

\def\R{\mathcal{R}}
\def\S{\mathcal{S}}
\def\PG{\mathrm{PG}}

\begin{document}

\title{Upper bounds on the smallest size of a saturating set in
projective planes and spaces of even dimension\thanks{The research of  D. Bartoli, M. Giulietti, S. Marcugini, and F.~Pambianco was
 supported in part by Ministry for Education, University
and Research of Italy (MIUR) (Project ``Geometrie di Galois e
strutture di incidenza'')
 and by the Italian National Group for Algebraic and Geometric Structures and their Applications
 (GNSAGA - INDAM).
 The research of A.A.~Davydov was carried out at the IITP RAS at the expense of the Russian
Foundation for Sciences (project 14-50-00150).}}
\date{}
\maketitle
\begin{center}
{\sc Daniele Bartoli}\\
{\sc\small Dipartimento di Matematica  e Informatica,
   Universit{\`a} degli Studi di Perugia}\\
{\sc\small Perugia, 06123, Italy}\\ \emph {E-mail address:} daniele.bartoli@unipg.it\medskip\\
{\sc Alexander A. Davydov}\\
{\sc\small Institute for Information Transmission Problems (Kharkevich institute)}\\
 {\sc\small Russian Academy of
 Sciences}\\ {\sc\small GSP-4, Moscow, 127994, Russian Federation}\\\emph {E-mail address:} adav@iitp.ru\medskip\\
 {\sc Massimo Giulietti, Stefano Marcugini, Fernanda Pambianco}\\
 {\sc\small Dipartimento di Matematica  e Informatica,
  Universit{\`a} degli Studi di Perugia}\\
 {\sc\small Perugia, 06123, Italy}\\
 \emph{E-mail address:} massimo.giulietti, stefano.marcugini, fernanda.pambianco@unipg.it
\end{center}
\begin{abstract}
In a projective plane $\Pi _{q}$ (not necessarily Desarguesian)
of order $q,$ a point subset $\S$ is saturating (or dense) if
any point of $\Pi _{q}\setminus \S$ is collinear with two points
in$~\S$. Modifying an approach of \cite{Nagy}, we proved the following upper bound
on the smallest size $ s(2,q)$ of a saturating set in $\Pi
_{q}$:
\begin{equation*}
s(2,q)\leq \sqrt{(q+1)\left(3\ln q+\ln\ln q +\ln\frac{3}{4}\right)}+\sqrt{\frac{q}{3\ln q}}+3.
\end{equation*}
The bound holds for all $q$, not necessarily large.

By using inductive constructions,
 upper bounds on the smallest size
of a saturating set in the projective space $\PG(N,q)$ with even dimension $N$ are obtained.

All the results are also stated in terms of linear covering codes.
\end{abstract}

\section{Introduction}

We denote by $\Pi _{q}$ a projective plane (not necessarily
Desarguesian) of order $q$ and by $\PG(2,q)$ the projective
plane over the Galois field with $q$ elements.

\begin{definition}
\label{def1_usual satur} A point set $\S\subset \Pi _{q}$ is
\emph{saturating} if any point of $\Pi _{q}\setminus \S$ is
collinear with two points in $\S$.
\end{definition}

Saturating sets are considered, for example, in \cite
{Bartocci,BDGMP_SatSetArxiv,BDGMP_SatSet_Petersb,BFMP-JG2013,BorSzTic,BrPlWi,DavCovCodeSatSetIEEE1995,DavCovRad2,DGMP_CovCodeNonBin_Pamporovo,%
DGMP_CovCodeR23_Petersb2008,DGMP-AMC,DMP-JCTA2003,DavOst-EJC,DavOst-IEEE2001,FainaGiul,GacsSzonyi,Giul-plane,Giul2013Survey,GiulTor,Janwa1990,Kiss_Cayley,%
Kovacs,MP_Austr2003,Nagy,Ughi};
see also the references therein. It should be noted that
saturating sets are also called ``saturated
sets'' \cite
{DavCovCodeSatSetIEEE1995,DavCovRad2,Janwa1990,Kovacs,Ughi}, ``spanning
sets'' \cite{BrPlWi}, ``dense sets''
\cite{Bartocci,BorSzTic,FainaGiul,GacsSzonyi,Giul-plane,GiulTor}, and ``1-saturating sets''
\cite{DGMP_CovCodeNonBin_Pamporovo,DGMP_CovCodeR23_Petersb2008,DGMP-AMC,DMP-JCTA2003,DavOst-EJC}.

A particular kind of saturating sets in a projective plane is \emph{complete arcs}.  An arc is a set
of points no three of which are collinear. An arc is said to be complete if it cannot be extended to a large arc;
see \cite{BDFKMP-PIT2014,BDFKMP_ComplArc_JG2016,BFMP-JG2013,FainaGiul,Giul2013Survey,KV} and the references therein.

The homogeneous coordinates of the points of a saturating set
of size $k$ in $PG(2,q)$ form a parity check matrix of a
$q$-ary linear code with length $k,$ codimension 3, and
covering radius 2. For an introduction to covering codes see
\cite{Handbook-coverings,CHLS-bookCovCod}. An online
bibliography on covering codes is given in \cite{LobstBibl}.

The main problem in this context is to find small saturating
sets (i.e. short covering codes).

Denote by $s(2,q)$  \emph{the
smallest size of a saturating set in} $\Pi _{q}$.

Let $s_D(2,q)$ be
the smallest size of a saturating set in the Desarguesian plane $\PG(2,q)$.

Let $t_2(2,q)$ be
the smallest size of a complete arc in  $\PG(2,q)$.

Clearly,
\begin{align*}
    s_D(2,q)\le t_2(2,q).
\end{align*}

The trivial lower bound is
\begin{align*}
    s(2,q),s_D(2,q),t_2(2,q)>\sqrt{2q}+1.
\end{align*}

Saturating sets in $\PG(2,q)$ obtained by
\emph{algebraic constructions or computer search} can be found in
\cite
{Bartocci,BorSzTic,BFMP-JG2013,BrPlWi,DavCovCodeSatSetIEEE1995,DavCovRad2,DGMP_CovCodeNonBin_Pamporovo,DGMP_CovCodeR23_Petersb2008,%
DGMP-AMC,DMP-JCTA2003,DavOst-EJC,DavOst-IEEE2001,FainaGiul,Giul-plane,Giul2013Survey,GiulTor,Kiss_Cayley,MP_Austr2003,SzonyiPhD,SzonyiSurvey,Ughi}.

For  $\PG(2,q)$ with $q$ non-prime, in the literature there are a few algebraic constructions
of relatively small saturating sets providing, for instance, the following upper bounds:
\begin{align*}
&   s_D(2,q)<3\sqrt{q}-1&&if&&q=(q')^2&\mbox{\cite{DavCovCodeSatSetIEEE1995}};\\
&s_D(2,q)<2\sqrt{q}+2\sqrt[4]{q}+2&&if&&q=(q')^4&\mbox{\cite{DGMP_CovCodeNonBin_Pamporovo,DGMP_CovCodeR23_Petersb2008,DGMP-AMC,Kiss_Cayley}};\\
&s_D(2,q)<2\sqrt{q}+2\sqrt[3]{q}+2\sqrt[6]{q}+2&&if&&q=(q')^6,~q'~\mbox{prime},~q'\le73&\mbox{\cite{DGMP_CovCodeNonBin_Pamporovo,DGMP_CovCodeR23_Petersb2008,DGMP-AMC}};\\
&s_D(2,q)<2\sqrt[m]{q^{m-1}}+\sqrt[m]{q}&&if&&q=(q')^m,~m\ge2&\text{\cite{DavOst-EJC,Giul-plane}}.
\end{align*}

Saturating sets of size approximately $Cq^\frac{3}{4}$, with $C$ a constant independent on $q$, have been explicitly described in several
papers; see \cite{Bartocci,BorSzTic,GiulTor,SzonyiPhD,SzonyiSurvey}.

In \cite{Giul-plane}, algebraic constructions of saturating sets in
$\PG(2,q)$ of size about $3q^{ \frac{2}{3}}$ are proposed
and the following bounds are obtained (here $p$ is prime):
\begin{align}\label{eq1_Giul2007}
&s_D(2,q)<\frac{2q}{p^{\,t}}+\frac{(p^{\,t}-1)^{2}}{p-1}+1&&if&&~q=p^{m},~m\geq 2t;\\
&s_D(2,q)<\frac{2}{p}\sqrt[3]{(qp)^{2}}+\frac{\sqrt[3]{(qp)^{2}}-2\sqrt[3]{qp}+1}{p-1}
+1&&if&&~q=p^{3t-1};\notag\\
&s_D(2,q)<\min\limits_{v=1,\ldots ,2t+1}\Phi(t,p,v)&&if&&~q=p^{2t+1},\notag
\end{align}
$$\text{where }\Phi(t,p,v)=\left\{ (v+1)p^{t+1}+\frac{(p^{t}-1)^{2v}}{
(p-1)^{v}(p^{2t+1}-1)^{(v-1)}}+2\right\}.$$
For many triples $(t,p,v)$, constructions of \eqref{eq1_Giul2007} provide relatively small saturating sets, see \cite{Giul-plane}.

In \cite{BDFKMP_ComplArc_JG2016}, by computer search in a wide region of $q$, the following upper bounds for the smallest sizes of complete arcs in $\PG(2,q)$ are obtained:
\begin{align}\label{eq1_comp_JG}
&s_D(2,q)\le t_2(2,q)<0.998\sqrt{3q\ln q}&&\mbox{for}&&7\le q\le 160001;\\
&s_D(2,q)\le t_2(2,q)<1.05\sqrt{3q\ln q}&&\mbox{for}&&160001< q\le 301813.\notag
\end{align}
For $q\le160001$ greedy algorithms are used while for $160001< q\le 301813$ the algorithm with fixed order of points (FOP) is applied.

In \cite{BDFKMP-PIT2014}, for $\PG(2,q)$  an iterative step-by-step construction of complete
arcs, which adds a new point in each step, is considered. As an example, it is noted
the \emph{step-by-step greedy algorithm} that in every step adds to the arc a point providing the maximal possible (for the
given step) number of new covered points. For more than half of steps of the iterative process, an
estimate for the number of new covered points in every step is proved. A natural (and well-founded)
conjecture is made that the estimate holds for the other steps too. \emph{Under this conjecture}, the following upper
bound on the smallest size  of a complete arc in  $\PG(2,q)$ is obtained.
\begin{align}
\textbf{conjectural bound: } s_D(2,q)\le t_2(2,q)<\sqrt{q}\sqrt{3\ln q+\ln\ln q+\ln 3}+\sqrt{\frac{q}{3\ln q}}+3.\label{eq2_conj_PIT}
\end{align}
Note also that in \cite{BDFKMP-PIT2014} a \emph{truncated iterative step-by-step process} is considered. The process stops when  the
number of uncovered points attains some (\emph{a priori arbitrary assigned}) value. Then this value is summarized with the number of steps, executed before stopping of the iterative process. The estimate \eqref{eq2_conj_PIT} is obtained when the value, a priori assigned to stop the process, is $\sqrt{\frac{q}{3\ln q}}$; it implies that
the number of the steps, executed before stopping of the step-by-step process, is $\sqrt{q}\sqrt{3\ln q+\ln\ln q+\ln 3}$.

Surveys and results of \emph{probabilistic constructions} for geometrical objects can be
found in \cite {BDGMP_SatSetArxiv,BDGMP_SatSet_Petersb,BFMP-JG2017,BorSzTic,GacsSzonyi,KV,Kovacs,Nagy}; see also the
references therein.

In \cite{BorSzTic}, by using a modified probabilistic approach
introduced in \cite{Kovacs}, the following upper
bound for an arbitrary (not necessarily Desarguesian) plane is proved:
\begin{equation}
s(2,q)<3\sqrt{2}\sqrt{q\ln q}<5\sqrt{q\ln q}.  \label{eq1_BST}
\end{equation}
In \cite{BDGMP_SatSetArxiv}, see also \cite{BDGMP_SatSet_Petersb}, by probabilistic methods different from these in \cite{BorSzTic,Kovacs} the upper bound
\begin{equation}\label{eq1_BDGMP_2015}
s(2,q)\leq 2\sqrt{(q+1)\ln (q+1)}+2\thicksim 2\sqrt{q\ln q}
\end{equation}
is obtained.

In \cite{Nagy}, Z. Nagy obtained the following bound
\begin{equation}\label{eq1_Nagy}
s(2,q)\leq (\sqrt{3}+o(1))\sqrt{q\ln q)}.
\end{equation}
The proof of \eqref{eq1_Nagy} is given in \cite{Nagy} by two approaches: probabilistic and algorithmic. In the both approaches, starting with some stage of the proof,
it is assumed (by the context) that \emph{$q$ is large enoug}h.

The algorithmic
approach in \cite{Nagy} considers an original  step-by-step greedy algorithm and obtains estimates for the number of new covered points in every step of the algorithm.
In order to obtain the bound, the
iterative process stops after executing of $\left\lceil\sqrt{3q\ln q}\,\right\rceil$ steps. It is proved in \cite{Nagy}, that in this case the number of uncovered points is not greater than $\sqrt{q}$. Then the half of the number of uncovered points is summarized with the number of executed steps. As the result of the algorithmic proof of \cite{Nagy}, the following form of the bound can be derived.
\begin{equation}\label{eq1_Nagy_finite}
s(2,q)\leq \left\lceil\sqrt{3q\ln q}\,\right\rceil+\left\lceil\frac{1}{2}\sqrt{q}\right\rceil\leq \sqrt{3q\ln q}+\frac{1}{2}\sqrt{q}+2,\quad q~\text{large enough}.
\end{equation}

In some sense the algorithmic approach of \cite{Nagy} is close to consideration of bounds in \cite{BDFKMP-PIT2014}.  But in \cite{BDFKMP-PIT2014} the number of steps, executed before stopping of the iterative process, depends on a priori assigned number of uncovered points. At the same time, in \cite{Nagy} the
iterative process always stops after executing of $\left\lceil\sqrt{3q\ln q}\,\right\rceil$ steps. Of course, it must be noted that
in \cite{BDFKMP-PIT2014} the bound is conjectural (as the estimates are not proved for all steps of the iterative greedy process)
 whereas in \cite{Nagy} the bound is proved. Note also that problems considered in \cite{BDFKMP-PIT2014}  and \cite{Nagy} are close but not the same (small complete arcs in \cite{BDFKMP-PIT2014}  and small saturating sets in \cite{Nagy}).

\emph{In this paper}, we modify the algorithmic approach of \cite{Nagy} so that the final formula holds for an arbitrary $q$ (\emph{not necessarily large}) and, moreover,
the value of a \emph{new bound is smaller}  than in \eqref{eq1_Nagy_finite}, see \eqref{eq2_Delta}--\eqref{eq2_limitDelta/sqrtqlnq}.

Our main results is Theorem \ref{th1}.
\begin{theorem}
\label{th1} For the smallest size $s(2,q)$ of a saturating set
in a projective plane (not necessarily Desarguesian) of order $q$ \emph{(not necessarily large)} the following upper bound
holds:
\begin{equation}
s(2,q)\leq \sqrt{(q+1)\left(3\ln q+\ln\ln q +\ln\frac{3}{4}\right)}+\sqrt{\frac{q}{3\ln q}}+3.
\label{eq1_satsetsize}
\end{equation}
\end{theorem}

Note that modifying the algorithmic approach of \cite{Nagy}, we (similarly to \cite{BDFKMP-PIT2014}) stop the iterative process when  the
number of uncovered points attains a priori assigned value, $\xi$ say. If $\xi=1$ we obtain the bound coinciding with   \eqref{eq1_BDGMP_2015}; if $\xi=\sqrt{q}$ we obtain the bound coinciding with   \eqref{eq1_Nagy_finite}, see Remark \ref{rem2}. Finally, if $\xi=\sqrt{\frac{4q}{3\ln q}}$ we get the bound    \eqref{eq1_satsetsize}.

\begin{remark}
It is interesting that the main term $\sqrt{3q\ln q}$ is the same in the bounds \eqref{eq1_comp_JG}, \eqref{eq2_conj_PIT} for complete arcs and \eqref{eq1_Nagy}--\eqref{eq1_Nagy_finite},
\eqref{eq1_satsetsize} for saturating sets.
\end{remark}

Theorem \ref{th1} can be expressed in terms of \emph{covering codes}.

The \emph{length function} $\ell (R,r,q)$ denotes the smallest
length of a $ q $-ary linear code with covering radius $R$ and
codimension~$r$; see \cite
{Handbook-coverings,BrPlWi,CHLS-bookCovCod}.

Theorem~\ref{th1}
can be read as follows.

\begin{corollary}
The following upper bound on the length function holds.
\begin{equation*}
\ell (2,3,q)\leq \sqrt{(q+1)\left(3\ln q+\ln\ln q +\ln\frac{3}{4}\right)}+\sqrt{\frac{q}{3\ln q}}+3.
\end{equation*}
\end{corollary}

Let $\PG(N,q)$ be the $N$-dimensional projective space over the
Galois field of $q$ elements.

\begin{definition}
A point set $\S\subset \PG(N,q) $ is \emph{saturating} if any
point of $\PG(N,q)\setminus \S$ is collinear with two points in
$\S$.
\end{definition}

A particular kind of saturating sets in a projective space is \emph{complete caps}.  A cap is a set
of points no three of which are collinear. A cap is said to be complete if it cannot be extended to a large cap.

Let $[n,n-r]_{q}R$ be a linear $q$-ary code of length $n,$
codimension $r,$ and covering radius$~R.$ The homogeneous
coordinates of the points of a saturating set with size $n$ in
$\PG(r-1,q),$ form a parity check matrix of an $[n,n-r]_{q}2$
code.

 Results on saturating sets in $\PG(N,q)$ and the corresponding covering codes can be found  in \cite
{BFMP-JG2017,Handbook-coverings,BrPlWi,CHLS-bookCovCod,DavCovCodeSatSetIEEE1995,DavCovRad2,DGMP_CovCodeNonBin_Pamporovo,%
DGMP_CovCodeR23_Petersb2008,DGMP-AMC,DMP-JCTA2003,DavOst-EJC,DavOst-IEEE2001,Giul2013Survey,Janwa1990,Ughi}
and the references therein.

Let $s(N,q)$ be \emph{the smallest size of a saturating set in}
$\PG(N,q),$ $N\geq 3$.

In terms of covering codes, we recall the
equality
$$s(N,q)=\ell (2,N+1,q).$$

The trivial lower bound for $s(N,q)$ is
\begin{align*}
    s(N,q)>\sqrt{2}q^\frac{N-1}{2}.
\end{align*}
Constructions of saturating sets (or the corresponding covering codes) whose
size is close to this lower bound are only known for $N$ odd, see \cite{DavCovRad2,DGMP-AMC,Giul2013Survey} for survey. In particular,
in \cite[Theorem 9]{DavOst-IEEE2001}, see also \cite[Section 4.3]{DGMP-AMC}, the following bound is obtained by algebraic constructions:
\begin{align}
    s(N,q)=\ell (2,N+1,q)\le 2q^{\frac{N-1}{2}}+q^{\frac{N-3}{2}},\quad N=2t-1\ge3,~N\neq7,11,~q\ge7,~q\neq9,
\end{align}
where $t=2,3,5$, and $t\ge7$.

From (\ref{eq1_satsetsize}), by
using inductive constructions from \cite
{DavCovRad2,DGMP-AMC}, we obtained upper bounds on the
smallest size of a saturating set in the $N$-dimensional
projective space $\PG(N,q)$ with $N$ even; see Section
\ref{sec_space}. In many cases these bounds are better than the
known ones.

The paper is organized as follows. In Section
\ref{sec_1-satur}, we deal with upper bounds on the smallest
size of a saturating set in a projective plane.  In Section \ref{sec_space}, bounds
for saturating  sets in the projective
space $\PG(N,q)$ are obtained.

\section{A modification of Nagy's approach for upper bound on the smallest size of a saturating set in a
projective plane\label{sec_1-satur}}

Assume that in $\Pi_q$ a saturating set
 is constructed by a step-by-step algorithm adding one new point to the set in
every step.

Let $i>0$ be an integer. Denote by $\S_i$ the running set obtained after the
$i$-th step of the algorithm. A point $P$ of $\Pi_q\setminus\S_i$ is covered by $\S_i$ if $P$ lies on $t$-secant of $\S_i$ with $t\ge2$.
 Let $\R_i$ be the subset of  $\Pi_q\setminus\S_i$ consisting of points not covered by $\S_i$.

 In \cite{Nagy} the following ingenious \emph{greedy algorithm} is proposed. One takes the line $\ell$ skew to $\S_i$ such that the cardinality of intersection $|\R_i\cap\ell|$ is the minimal among all skew lines.
Then one adds to $\S_i$ the point on $\ell$ providing the greatest number of new covered points (in comparison with other points of $\ell$). As a result we obtain the set $\S_{i+1}$ and the corresponding set $\R_{i+1}$.

The following Proposition is proved in \cite{Nagy}.
\begin{proposition}\emph{\cite[Proposition 3.3, Proof]{Nagy}} It holds that
\begin{align}
   &|\R_{i+1}|\le|\R_i|\cdot\left(1-\frac{i(q-1)}{q(q+1)}\right).\label{eq2_Nagy_R_i+1}
\end{align}
\end{proposition}

Clearly, that always
\begin{align}\label{eq2_R2}
    \R_2=q^2.
\end{align}

Iteratively applying the relation \eqref{eq2_Nagy_R_i+1} to $\R_2=q^2$, we obtain for some $k$ the following:
\begin{align}\label{eq2_Delta}
|\R_{k+1}|\le q^2\prod_{i=2}^k\left(1-\frac{i(q-1)}{q(q+1)}\right).
\end{align}

We denote
\begin{align}\label{eq2_fqk}
    f_q(k)=\prod_{i=2}^k\left(1-\frac{i(q-1)}{q(q+1)}\right).
\end{align}

Similarly to \cite{BDFKMP-PIT2014}, we consider a \emph{truncated iterative process}. We will stop the iterative process when $|\R_{k+1}|\leq \xi  $ where $\xi
\geq 1$ is some value that we may \emph{assign arbitrary} to improve estimates.

By \cite[Lemma 2.1]{Nagy} after the end of the iterative process we can add at most $\lceil\,|\R_{k+1}|/2\rceil$ points to
the running subset $\S_{k+1}$ in order to get the final saturating set $\S$.

The size $s$ of the obtained set $\S$ is
\begin{equation}\label{eq2_general_bound}
s\leq k+1+\left\lceil\frac{\xi}{2}\right\rceil \text{ under condition }q^2f_q(k)\leq \xi.
\end{equation}

Using the inequality $1-x\leq e^{-x}$ we obtain that
\begin{equation*}
f_q(k)<e^{-\sum_{i=2}^{k}\frac{i(q-1)}{q^{2}+q}}=e^{-\frac{(k^{2}+k-2)(q-1)}{
2(q^{2}+q)}},
\end{equation*}
which implies
\begin{equation}
f_q(k) <e^{-\frac{(k^{2}+k-2)(q-1)}{2(q^{2}+q)}}<e^{-\frac{k^{2}}{2q+2}},
\label{eq2_fq(k)_estimate}
\end{equation}
provided that
\begin{equation*}
\frac{(k^2+k-2)(q-1)}{q}>k^2
\end{equation*}
or, equivalently,
\begin{align*}
    \frac{k^2}{k-2}<q-1,
\end{align*}
\begin{equation}\label{eq2_k<}
k<q-4.
\end{equation}

\begin{lemma}
\label{lem3_basic}  Let $\xi\ge1$ be a fixed value independent of $k$. The value
\begin{equation}
k\geq \left\lceil\sqrt{2(q+1)}\sqrt{\ln \frac{q^2}{\xi }}\,\,\right\rceil  \label{eq2_k<=}
\end{equation}
satisfies  inequality $q^2f_q(k)\leq \xi$.
\end{lemma}

\begin{proof}
By \eqref{eq2_fq(k)_estimate}, to provide $q^2f_q(k)\leq \xi$ it
is sufficient to find $k$ such that
\begin{equation*}
e^{-\frac{k^{2}}{2q+2}}< \frac{\xi }{q^2}.
\end{equation*}
\end{proof}

\begin{theorem}
\label{th2_general bound_converse}In a plane $
\Pi_q$ it holds that
\begin{equation}
s(2,q)\leq \sqrt{2(q+1)}\sqrt{\ln \frac{
q^2}{\xi }}+\frac{\xi}{2}+3,~~\xi \geq 1,  \label{eq2_gen bound}
\end{equation}
where $\xi $ is an arbitrarily chosen value.
\end{theorem}

\begin{proof}
The assertion follows from \eqref{eq2_general_bound} and \eqref{eq2_k<=}.
\end{proof}

     We consider
the function of $\xi$ of the form
\begin{equation*}
\phi(\xi)=\sqrt{2(q+1)}\sqrt{\ln \frac{
q^2}{\xi }}+\frac{\xi}{2}+3.  \label{eq2_phi}
\end{equation*}
Its derivative by $\xi$ is
\begin{equation*}
\phi'(\xi)=\frac{1}{2}-\frac{1}{\xi}\sqrt{\frac{q+1}{2\ln\frac{q^{2}}{\xi }}}.  \label{eq2_der phi}
\end{equation*}
Put $\phi'(\xi)=0$. Then it is easy to see that
\begin{align}\label{eq2_xi^2}
    \xi^2=\frac{q+1}{\ln q-\frac{1}{2}\ln \xi}.
\end{align}
We find $\xi$ in the form $
\xi=\sqrt{\frac{q+1}{c\ln q}}$. By \eqref{eq2_xi^2},
\begin{align*}
    c=1-\frac{\ln(q+1)}{4\ln q}+\frac{\ln c+\ln\ln q}{4\ln q}.
\end{align*}
For simplicity, we choose $c\approx\frac{3}{4}$ and put
\begin{align}\label{eq2_xi_choice}
\xi =\sqrt{\frac{4q}{3\ln q}}.
\end{align}

Now, substituting $
\xi =\sqrt{\frac{4q}{3\ln q}}
$ in \eqref{eq2_gen bound}, we obtain Theorem \ref{th1}.

\begin{remark}\label{rem2}\begin{description}
                \item[(i)] Let $\xi=1$. From \eqref{eq2_gen bound} we have
    \begin{align}\label{eq2_xi=1}
       s(2,q)\le2\sqrt{(q+1)\ln q}+3,
    \end{align}
that practically coincides with  bound \eqref{eq1_BDGMP_2015} from \cite{BDGMP_SatSetArxiv,BDGMP_SatSet_Petersb}.
                \item[(ii)]  Let $\xi=\sqrt{q}$. From \eqref{eq2_gen bound} we obtain the estimate
    \begin{align}\label{eq2_xi=sqrt_q}
       s(2,q)\le\sqrt{3(q+1)\ln q}+\frac{1}{2}\sqrt{q}+3
    \end{align}which  practically coincides with Nagy's bound \eqref{eq1_Nagy_finite}.
However, as it is noted below, the value $\xi =\sqrt{\frac{4q}{3\ln q}}$ gives a better estimate than \eqref{eq2_xi=sqrt_q}.
              \end{description}
\end{remark}

We denote the difference
\begin{align*}
    \Delta(q)=\sqrt{3q\ln q}+\frac{1}{2}\sqrt{q}+2-\left(\sqrt{(q+1)\left(3\ln q+\ln\ln q +\ln\frac{3}{4}\right)}+\sqrt{\frac{q}{3\ln q}}+3\right).
\end{align*}
It can be shown (e.g. by consideration of the corresponding derivations) that
\begin{align}\label{eq2_Delta}
    \Delta(q)>0~\text{ for }~q\ge919,
\end{align}
and, moreover, $\Delta(q)$ and $\frac{\Delta(q)}{\sqrt{q}}$ are increasing functions of $q$.
For illustration, see Fig.\,\ref{fig_Delta} where the top dashed-dotted black curve shows $\Delta(q)$ while the bottom solid red curve $\sqrt{\frac{q}{7}}$ is given for comparison.
\begin{figure}[h]
\includegraphics[width=\textwidth]{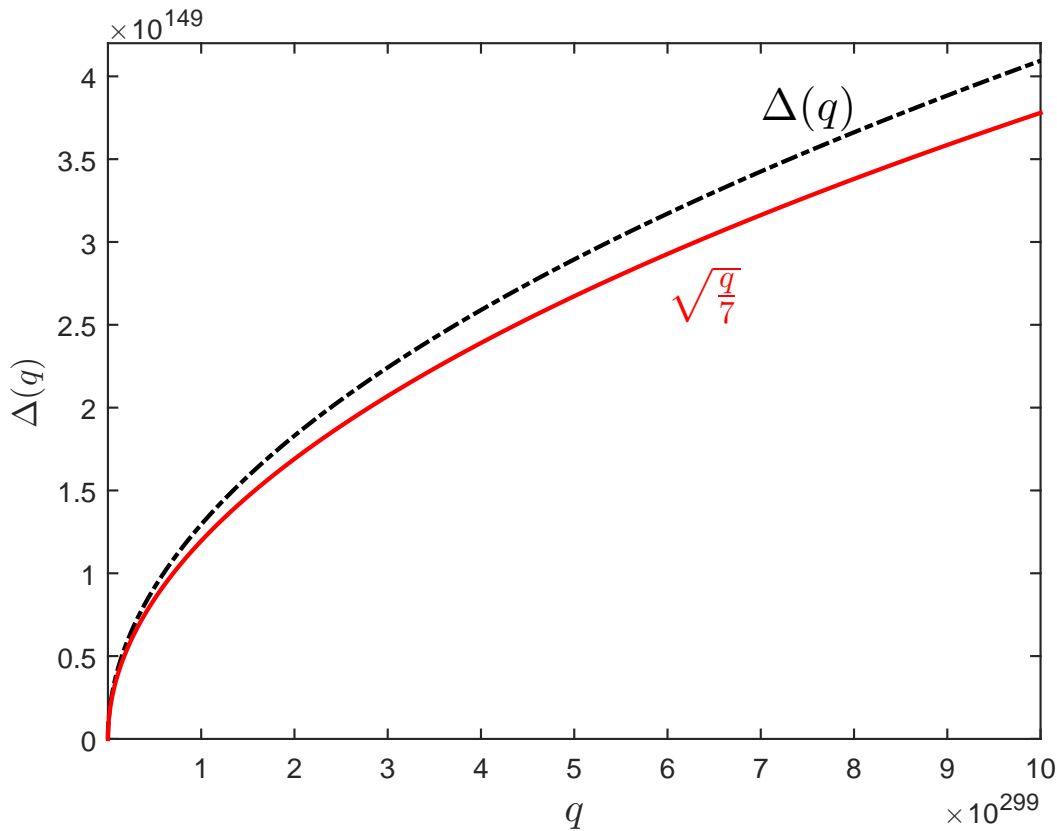}
\caption{The difference $\Delta(q)$ (\emph{top dashed-dotted black curve}) vs $\sqrt{\frac{q}{7}}$ (\emph{bottom solid red curve})}\label{fig_Delta}
\end{figure}

Note also that
\begin{align*}
&    \frac{\Delta(q)}{\sqrt{q}}\approx\sqrt{3\ln q}+\frac{1}{2}-\sqrt{3\ln q+\ln\ln q}-\frac{1}{\sqrt{3\ln q}}\,,\\
 &   \frac{\Delta(q)}{\sqrt{q\ln q}}\approx\sqrt{3}+\frac{1}{2\sqrt{\ln q}}-\sqrt{3+\frac{\ln\ln q}{\ln q}}-\frac{1}{\sqrt{3}\ln q}\,,
\end{align*}
whence
\begin{align}
\lim_{q\rightarrow\infty}\frac{\Delta(q)}{\sqrt{q}}=\frac{1}{2},\label{eq2_limitDelta/sqrtq}\\
\lim_{q\rightarrow\infty}\frac{\Delta(q)}{\sqrt{q\ln q}}=0.\label{eq2_limitDelta/sqrtqlnq}
\end{align}

\section{Upper bounds on the smallest size of a saturating set in the
projective space $\PG(N,q)$, $N$ even\label{sec_space}}

In further we use the results of \cite{DavCovRad2,DGMP-AMC} that give the following inductive construction.
\begin{proposition}\label{prop3}
 \emph{\cite[Example 6]{DavCovRad2} \cite[Theorem 4.4]{DGMP-AMC}}
  Let  exist an $ [n_{q},n_{q}-3]_{q}2$ code with $n_q<q$.
 Then, under condition $q+1\geq 2n_q$, there is an infinite family of
$[n,n-r]_{q}2$ codes with $r=2t-1\geq
5,~r\neq 9,13,~n=n_{q}q^{t-2}+2q^{t-3}$, where $t=3,4,6$, and $t\ge8$.
For $r=9,13,$ it holds that $n=n_{q}q^{t-2}+2q^{t-3}+q^{t-4}+q^{t-5}$.
\end{proposition}

Now due to one-to-one correspondence between covering codes and saturating sets we obtain the corollary from Theorem \ref{th1} and Proposition \ref{prop3}.
We denote
\begin{align*}
    \Upsilon(q)=\sqrt{(q+1)\left(3\ln q+\ln\ln q +\ln\frac{3}{4}\right)}+\sqrt{\frac{q}{3\ln q}}+3.
\end{align*}

\begin{corollary}
\label{cor_space}For the smallest size $s(N,q)$ of a
saturating set in the projective space $\PG(N,q)$ and for the length function $\ell (2,N+1,q),$
the following upper bounds hold:\begin{description}
                                  \item[(i)] \begin{align}
s(N,q)=\ell (2,N+1,q)\leq \Upsilon(q)\cdot q^{\frac{N-2}{2}}+2q^{\frac{N-4}{2}},\quad N=2t-2\geq
4,\quad N\neq 8,12,\label{eq3_SatSpace}
\end{align}
where  $t=3,4,6$, and $t\ge8$,  $q\geq79$.
                                  \item[(ii)] \begin{align}
s(N,q)=\ell (2,N+1,q)\leq\Upsilon(q)\cdot q^{\frac{N-2}{2}}+2q^{\frac{N-4}{2}}+q^{\frac{N-6}{2}}+q^{\frac{N-8}{2}},\quad N=8,12. \label{eq3_SatSpace2}
\end{align}
                                \end{description}
\end{corollary}

\begin{proof}
By Theorem \ref{th1}, in $\PG(2,q)$ there is a saturating set with size $n_q=\Upsilon(q)$.
 From the corresponding $[n_{q},n_{q}-3]_{q}2$
code,
one can obtain an $[n,n-r]_{q}2$ codes with parameters as in Proposition \ref{prop3}. The condition $q+1\geq 2n_q$ holds for $q\ge79$.
\end{proof}

Surveys of the known $[n,n-r]_{q}2$ codes and saturating sets
in $\PG(N,q)$ with $N$ even can be found in \cite{DavCovRad2,DGMP-AMC,Giul2013Survey}.
In many cases bounds (\ref{eq3_SatSpace}), \eqref{eq3_SatSpace2} is better
than the known ones.

\end{document}